\documentclass[12pt,reqno]{amsart}
\usepackage{graphicx} 
\usepackage[english]{babel}
\usepackage[utf8]{inputenc}
\usepackage[T1,T2A]{fontenc}

\usepackage{amsmath}
\usepackage{amsfonts}
\usepackage{amssymb}
\usepackage{amsthm}

\usepackage{pdfpages}
\usepackage[bookmarks,hypertexnames=true]{hyperref}\hypersetup{colorlinks=true,linkcolor=black,citecolor=black,pdfauthor={Andr\'es Ahumada G\'omez},pdfcenterwindow,pdftitle={Equivariant Smoothing Processes on Currents and Spaces with Bounded Curvature - Andr\'es Ahumada G\'omez},pdfdisplaydoctitle,pdftoolbar=false,pdfcreator={Andr\'es Ahumada G\'omez},hyperindex=true}

\usepackage{mathrsfs}
\usepackage{enumerate} 
\usepackage{accents}
\usepackage[all,cmtip]{xy}
\usepackage{scalerel}
\usepackage{comment}

\usepackage{microtype}
\usepackage{ae,aecompl}

\usepackage{anysize}
\marginsize{2cm}{2cm}{2cm}{2cm}

\newtheorem{teo}{Theorem}[]
\newtheorem{lem}[teo]{Lemma}
\newtheorem{pro}[teo]{Proposition}

\newtheorem{obss}[teo]{Remark}
\theoremstyle{remark}

\newtheorem*{ack}{Acknowledgements}

\newcommand{\Esf}{\mathbb{S}}
\newcommand{\N}{\mathbb{N}}

\newcommand{\R}{\mathbb{R}}
\newcommand{\B}{\mathbb{B}}
\newcommand{\G}{\mathcal{G}}

\newcommand{\Zoo}{\mathbf{Z}}
\newcommand{\Zo}{\mathcal{Z}}
\newcommand{\ZZo}{\tilde{\Zoo}}

\newcommand{\Mr}{\mathfrak{M}}
\newcommand{\hr}{\mathfrak{h}}

\newcommand{\Hr}{\mathcal{H}}
\newcommand{\Hrb}{\mathbf{H}}
\newcommand{\Hrt}{\tilde{\Hrb}}
\newcommand{\gr}{\texttt{g}}
\newcommand{\s}{^{*}}
\newcommand{\sfo}{_{*}}
\newcommand{\ud}{\, d}
\newcommand{\om}{\omega}
\newcommand{\de}{\delta}

\newcommand{\e}{\varepsilon}
\newcommand{\vf}{\varphi}
\newcommand{\la}{\lambda}
\newcommand{\al}{\alpha}

\newcommand{\si}{\sigma}
\newcommand{\ga}{\gamma}
\newcommand{\dga}{\dot{\gamma}}

\newcommand{\Ddd}{\mathscr{D}}

\newcommand{\M}{\mathcal{M}}

\DeclareMathOperator{\id}{Id}

\DeclareMathOperator*{\esssup}{ess\,sup}
\DeclareMathOperator*{\essinf}{ess\,inf}
\DeclareMathOperator{\spt}{spt}

\usepackage{graphicx}

\usepackage{lineno}

\title[Equivariant Smoothing Processes]{Equivariant Smoothing Processes on Currents and Spaces with Bounded Curvature}
\author[A.~Ahumada G\'omez]{Andr\'es Ahumada G\'omez}
\address[Ahumada G\'omez]{Facultad de Ciencias, Universidad Nacional Aut\'onoma de M\'exico, M\'exico}
\email{andres.ahumada.gomez@ciencias.unam.mx}
\thanks{Supported by CONACYT Doctoral Scholarship No. 734952}
\date{\today}


\begin{document}

\maketitle


\begin{abstract}
    We introduce actions of a compact Lie group in two regularization processes: in De Rham's approximation process of currents on a smooth manifold by smooth currents, and in a smoothing operator of Riemannian metrics of metric spaces with bounded curvature. 
\end{abstract}

\section{Introduction}

Riemannian geometry creates an ideal framework to study the topology of manifolds using all the differentiable, metric and algebraic machinery. Within this environment, Toponogov proved that having a lower bound on the sectional curvature is equivalent to a triangle comparison process (see Theorem 12.2.2, ~\cite{petersen2016}). This way of thinking has been generalized to metric spaces with a notion of curvature beginning in the 1950's with the work by A. D. Aleksandrov in \cite{aleksandrov1957}. Such notion of curvature defines when a metric space has a curvature bounded either from below or from above in terms of comparison of geodesic triangles.

It is natural to study the case where we have both curvature bounds, namely \emph{spaces with bounded curvature}, a combination of Alexandrov spaces (lower bound) and CAT spaces (upper bound). In 1975, V. N. Berestovskii proved in \cite{berestovski1976} that having both bounds is very restrictive and these spaces are actually Riemannian manifolds with low regularity. 

In 1983, I. G. Nikolaev improved the regularity by giving a $C^{3,\al}$ differentiable structure and a $C^{1,\al}$ Riemannian tensor to the spaces with bounded curvature in \cite{nikolaev1983} and \cite{nikolaev19832}.

Finally, in 1991 Nikolaev proved in \cite{nikolaev1991} that we can approximate a space with bounded curvature by smooth Riemannian manifolds controlling the sectional curvature:

  \begin{teo}[Nikolaev, 1991]\label{teoa}
   	Let $(\M,d(\textbf{\gr}_{0}))$ be a space with bounded curvature with $d$ its metric and $\textbf{\gr}_{0}$ the induced Riemannian metric by this. Then, on the differentiable manifold $\M$, one can define a sequence of infinitely differentiable Riemannian metrics $\{\textbf{\gr}_{m}\}_{m=1}^{\infty}$ having the following properties:
	\begin{enumerate}
		\item The metric spaces $(\M,d(\textbf{\gr}_{m}))$ converge in the Lipschitz sense to the metric space $(\M,d(\textbf{\gr}_{0}))$.
		\item The following estimates hold for the limits of curvature:
		\begin{equation*}
			\limsup_{m\to\infty}\bar{k}_{m}(\M)\leq\bar{k}_{0}(\M)\text{ and }\liminf_{m\to\infty}\underaccent{\bar}{k}_{m}(\M)\geq\underaccent{\bar}{k}_{0}(\M),
		\end{equation*}
	where $\bar{k}_{l}(\M)$ and $\underaccent{\bar}{k}_{l}(\M)$ denote the upper and lower limits of sectional curvature of the spaces $(\M,\textbf{\gr}_{l})$, $l\in\N$. 
	\end{enumerate}
  \end{teo}
  
    The proof of Theorem \ref{teoa} is an emulation of the proof of the following approximation procedure of currents by De Rham in \cite{derham1984} without the curvature control.

  \begin{teo}[De Rham's Approximation Theorem]\label{teoderham}
      On a manifold $M$, we can construct a linear operator $\Zo$ on the topological dual space of the space of compactly supported forms, depending on positive parameters  $\e_{1},\e_{2},\ldots$ which are finite or infinite in number according as $M$ is compact or not respectively, which have the following properties:
	\begin{enumerate}[(a)]
		\item If $T$ is an $m$-dimensional current in $M$, then $\Zo T$ is an $m$-dimensional current. 
		\item The support of $\Zo T$ is contained in any given neighborhood of the support of $T$ provided that the parameters $\e_{i}$ are sufficiently small.
		\item $\Zo T$ is $C^{\infty}$.	
		\item If each $\e_{i}$ tends to zero, $\Zo T $ converges weakly to $T$.
	\end{enumerate}
  \end{teo}

The goal of this paper is to introduce symmetries in processes described in Theorem \ref{teoa} and Theorem \ref{teoderham} considering that the isometry group of a space with bounded curvature is a Lie group (see \cite{myers_steenrod1939, fukaya_yamaguchi1994}) by proving the following two theorems.

\begin{teo}[Equivariant De Rham's Approximation Theorem]\label{tedr}
    Let $\G$ be a compact Lie group acting on a manifold $\M$ and let $T$ be an $m$-current on $\M$. If $T$ is $\G$-invariant, then we can construct a linear operator $\Zo_{\G}$ that depends on a number $\e>0$ such that $\Zo_{\G}T$ has the following properties:
	\begin{enumerate}[(a)]
		\item $\Zo_{\G}T$ is an $m$-current of class $C^{\infty}$.
		\item $\Zo_{\G}T$ is a $\G$-invariant current.
		\item If $\e$ tends to zero, $\Zo_{\G}T$ converges weakly to $T$.
	\end{enumerate}
\end{teo}
\begin{teo}[Equivariant Nikolaev's Approximation Theorem]\label{ENATeo}
    	Let $(\M,d(\textbf{\gr}_{0}))$ be a compact space with bounded curvature and a compact Lie group $\G$ acting on $\M$ by isometries. Then, on the differentiable manifold $\M$, one can define a sequence of infinitely differentiable Riemannian metrics $\{\textbf{\gr}_{k}\}_{k=1}^{\infty}$ having the following properties:
	\begin{enumerate}
		\item The Lie group $\G$ acts on $(\M,d(\textbf{\gr}_{k}))$ by isometries.
		\item The metric spaces $(\M,d(\textbf{\gr}_{k}))$ converge in the Lipschitz sense to the metric space $(\M,d(\textbf{\gr}_{0}))$.
		\item The following estimates hold for the limits of the curvature:
		\begin{equation*}
			\limsup_{k\to\infty}\bar{k}_{k}(\M)\leq\bar{k}_{0}(\M)\text{ and }\liminf_{k\to\infty}\underaccent{\bar}{k}_{k}(\M)\geq\underaccent{\bar}{k}_{0}(\M),
		\end{equation*}
		where $\bar{k}_{r}(\M)$ and $\underaccent{\bar}{k}_{r}(\M)$ denote the upper and lower limits of curvature of the spaces $(\M,\textbf{\gr}_{r})$, $r\in\N$.
	\end{enumerate}
\end{teo}
We notice that we called ``equivariant'' to Theorem \ref{tedr} and Theorem \ref{ENATeo} because the process defines an action in the limit using the action on the elements of the sequence, not because we have an equivariant map in their statements.

The main application of Theorem \ref{ENATeo} is that it provides an algorithmic procedure to extend results on actions of compact Lie groups on compact Riemannian manifolds to actions of compact Lie groups on spaces with bounded curvature. The first application of this procedure is to the equivariant sphere theorem.

\begin{teo}[Equivariant Sphere Theorem]\label{teoesf}
	Let $(\M,d)$ be a compact, simply connected space with bounded curvature of dimension $n\geq 4$ such that 
	\begin{equation}\label{eqesf}
		\frac{1}{4}<\underaccent{\bar}{K}(\M)\leq \bar{K}(\M)\leq 1.
	\end{equation}
	Assume that $\G$ is a compact Lie group and that there exists a group homomorphism \mbox{$\theta:\G\to\text{Isom}(\M)$}, i.e. $\G$ acts by isometries on $\M$. Then there exists a group homomorphism $\si:\G\to O(n+1)$ and a diffeomorphism $F:\M\to\Esf^{n}$ which is also equivariant; i.e. $F\circ\theta(g)=\si(g)\circ F$ for all $g\in\G$.
\end{teo}

Obtained following the same way to think we also have a couple of results about the topology of compact cohomogeneity one spaces with bounded curvature, spaces with bounded curvature that allow an action with orbits of codimension one. The first result is in the same spirit as the Theorem \ref{teoesf}:

\begin{teo}\label{symspa}
    Let $\M$ be an even dimensional compact, simply
connected cohomogeneity one $\G$-space with bounded curvature such that the lower bound is greater than zero. If $\G$ is a compact Lie group, then $\M$ is equivariantly diffeomorphic
to a compact rank one symmetric space.
\end{teo}

The final statement concerns the fundamental group, a Bonnet-Myers Theorem for compact cohomogeneity one spaces with bounded curvature: 
\begin{teo}\label{cohomogteo}
    A compact cohomogeneity one space with bounded curvature such that the lower bound is positive has finite fundamental group.
\end{teo}

The organization of the article is the following. In section 2 we prove the Equivariant De Rham's Approximation Theorem, then in section 3 we prove the Equivariant Nikolaev's Approximation Theorem. Finally, in section 4, we prove the applications such as the Equivariant Sphere Theorem, Theorem \ref{symspa} and Theorem \ref{cohomogteo}.

\begin{ack}
This work is part of the author’s Ph.D. thesis. The author would like to thank his advisor Fernando Galaz Garc\'{\i}a for the support throughout the development of this paper. Also want to thank Mauricio Che, Diego Corro, Jes\'{u}s N\'{u}\~{n}ez Zimbr\'{o}n and Oscar Alfredo Palmas Velasco for their very valuable comments about the first draft of the paper. 
\end{ack}

\section{Currents}
Recall that a $k$-current on a smooth manifold $\M$ is simply a continuous linear functional on the space of compactly supported smooth $k$-forms on $\M$. Currents may be viewed as a natural extension of integration over submanifolds. The basic theory of currents can be found in \cite{derham1984,federer96,morgan16}.

We start studying the case of currents by recalling the construction of two operators on currents due to De Rham (see \cite{derham1984} for a detailed account).

For a point $y$ in $\R^{n}$ we consider the translation $\tau_{y}\colon\R^{n}\to\R^{n}$ given by $\tau_{y}(x)=x+y$ and the homotopy $\tau\colon[0,1]\times\R^{n}\to\R^{n}$ given by $\tau(t,x):=\tau_{ty}(x)=x+ty$.

Now we choose the function $\psi\colon\R\to\R$ given by
\begin{equation*}
	\psi(t)=\left\lbrace \begin{array}{cc}
		\frac{1}{\la_{n}}\exp(\frac{t^{2}}{t^{2}-1}),&0\leq|t|<1,\\
		0,&1\leq |t|,
	\end{array} \right. 
\end{equation*}
where the constant $\la_{n}$ is chosen so that 
\begin{equation*}
	\int_{\R}\psi(t)\ud t=\int_{-1}^{1}\psi(t)\ud t=1.
\end{equation*}

For $\e>0$, we take $f_{\e}\colon\R^{n}\to\R$ given by 
\begin{equation*}
	f_{\e}(x)=\frac{1}{\e^{n}}\psi\left( \frac{\left\| x\right\|}{\e}\right).
\end{equation*}
This is a radial nonnegative $C^{\infty}$ function whose support is contained in $B(0,\e)$, the ball with center $0$ and radius $\e$, and 
\begin{equation*}
	\int_{\R^{n}}f_{\e}(x)\ud x=1.
\end{equation*} 

Putting everything together we can define an operator $\Zoo\colon \Ddd_{m}(\R^{n})\to\Ddd_{m}(\R^{n})$ as follows:
\begin{equation*}
	\Zoo T(\om):=\int_{\R^{n}}f_{\e}(y)\cdot(\tau_{y})\sfo(T)(\om)\ud y.
\end{equation*}

\begin{pro}[Proposition 1, Section 14, Chapter III, of \cite{derham1984}]\label{prosmooth}
	The linear operator $\Zoo$ has the following properties:
	\begin{enumerate}[\quad a)]
		\item If $T$ is an $m$-dimensional current in $\R^{n}$, then $\Zoo T$ is an $m$-dimensional current.
		\item $\Zoo T$ is $C^{\infty}$.
		\item The support of $\Zoo T$ is contained in the $\e$-neighborhood of the support of $T$. 
		\item If $\e$ tends to zero, $\Zoo T$ converges weakly to $T$, i.e., $\Zoo T(\om)\to T(\om)$ for every form $\om\in\Ddd^{m}(\R^{n})$.
	\end{enumerate} 
\end{pro}

The next step is to define a smoothing operator analogous to the one in Proposition \ref{prosmooth} in the case of manifolds (Theorem \ref{teoderham}), but we also use it in the equivariant case. We will transform the operator $\Zoo$ defined on $\R^{n}$ with the aid of a homeomorphism $h\colon \R^{n}\to\B^{n}$, where $\B^{n}$ denotes the unit open ball with center in the origin in $\R^{n}$.

Let $g\colon(0,1)\to(0,\infty)$ be a function given by
\begin{equation*}
	g(r)=\left\lbrace \begin{array}{cc}
		r&r\in\left( 0,1/3\right], \\
		\tilde{g}(r)&r\in\left[ 1/3,2/3\right], \\
		\exp\left(\frac{1}{(1-r)^{2}} \right) & r\in\left[ 2/3,1\right),  
	\end{array}\right.
\end{equation*}
where $\tilde{g}\colon[1/3,2/3]\to[1/3,\exp(9)]$ is such that $g$ is $C^{\infty}$ and $g'(r)>0$. This function $g$ is bijective and we denote by $g^{-1}\colon(0,\infty)\to(0,1)$ its inverse. Now we define $h\colon \R^{n}\to\B^{n}$ as follows
\begin{equation*}
	h(x)=\left\lbrace \begin{array}{cc}
		\frac{g^{-1}(\left\| x\right\|)}{\left\| x\right\|}\,x&x\neq 0, \\
		\lim\limits_{x\to\bar{0}} \hspace{7pt}\frac{g^{-1}(\left\| x\right\|)}{\left\| x\right\|}\,x& x=0,  
	\end{array}\right.
\end{equation*}
which is a $C^{\infty}$ diffeomorphism. Then we have a diffeomorphism $s_{y}\colon\R^{n}\to\R^{n}$ of class $C^{\infty}$ given by
\begin{equation}\label{aplis}
	s_{y}(x):=\left\lbrace \begin{array}{cc}
		h\circ\tau_{y}\circ h^{-1}(x)&x\in \B^{n},\\
		x& x\in \R^{n}\smallsetminus\B^{n}.
	\end{array} \right. 
\end{equation}

This function allows us to define an operator $\ZZo\colon \Ddd_{m}(\R^{n})\to \Ddd_{m}(\R^{n})$ as follows
\begin{equation*}
	\ZZo T(\om):=\int_{\R^{n}}f_{\e}(y)\cdot(s_{y})\sfo(T)(\om)\ud y.
\end{equation*}

\begin{pro}[Proposition 2, Section 15, Chapter III, of \cite{derham1984}]\label{pro2}
	The linear operator $\ZZo$ has the following properties:
	\begin{enumerate}[\quad a)]
		\item If $T$ is an $m$-dimensional current in $\R^{n}$, then $\ZZo T$ is an $m$-dimensional current.
		\item The support of $\ZZo T$ is contained in the set 
		\begin{equation*}
			E(T,\e)=\bigcup_{y\in\R^n,\, \left| y\right| <\e}\spt((s_{y})\sfo T).
		\end{equation*}
		\item If $\e$ tends to zero, $\ZZo T$ converges weakly to $T$.
		\item $\ZZo T$ is $C^{\infty}$ on $\B^n$ and $\ZZo T=T$ in $\R^{n}\smallsetminus\bar{\B}^{n}$. If $T$ is $C^{r}$ in a neighborhood of a boundary point of $\B^{n}$, then $\ZZo T$ is also $C^{r}$ on a neighborhood of this point.
	
	\end{enumerate}
\end{pro}

It is time to introduce the symmetries in the process. We say that an $m$-current $T$ on $\M$ is $\G$-invariant if $(\al^g)\sfo T=T$ for all $g\in\G$, where $\G$ is a compact Lie group acting on $\M$ by diffeomorphisms, $\al:\G\times\M\to\M$ is such action and $\al^g:\M\to\M$ is the diffeomorphism given by $\al^g(x)=\al(g,x).$

Let us see what happens when we consider $\M$ as $\R^{n}$.

\begin{pro}
	Let $\G$ be a compact Lie group acting on $\R^{n}$ and let $T$ be a $m$-current. The operators $\Zoo$ and $\ZZo$ are $\G$-invariant.
\end{pro}
\begin{proof}
	This result is obtained observing that the following equations are satisfied for both operators:
	\begin{eqnarray*}
		(\al^g)\sfo \Zoo T (\om)=\Zoo T(\rho (\al^g)\s\om)&=&\int_{\R^{n}}f_{\e}(y)\cdot(\tau_{y})\sfo(T)(\rho (\al^g)\s\om)\ud y,\\
		&=&\int_{\R^{n}}f_{\e}(y)\cdot(\tau_{y})\sfo \,(\al^g)\sfo(T)(\om)\ud y,\\
		&=&\int_{\R^{n}}f_{\e}(y)\cdot(\tau_{y})\sfo(T)(\om)\ud y,\\
		&=&\Zoo T(\om),
	\end{eqnarray*}
where $\rho$ is a $C_{c}^{\infty}$ function equal to 1 in a neighborhood of the compact set 
\begin{equation*}
	\left( \left. \al^{g}\right| _{\spt(\Zoo T)}\right) ^{-1}(\spt(\om)).
\end{equation*}
\end{proof}

Before we consider the case of an arbitrary manifold, we have the following property about invariant currents.

\begin{pro}\label{proinv}
	Let $\G$ be a compact Lie group acting on a manifold $\M$ and let $T$ be a $\G$-invariant $m$-current on $\M$. Then
	\begin{equation*}
		\left| \left( \al^{g}\right)\sfo T(\om) \right| \leq C \left| T(\om)\right| 
	\end{equation*}
for every $\om\in \Ddd^{m}(\M)$.
\end{pro}
\begin{proof} We take $\om\in \Ddd^{m}(\M)$. Then
	\begin{equation*}
		\left| \left( \al^{g}\right)\sfo T(\om) \right|=\left|T \left( \rho \left( \al^{g} \right)\s \om \right) \right| \leq \left|T \left( \rho \left( \al^{g_{0}} \right)\s \om \right) \right| \leq C \left| T(\om)\right|.
	\end{equation*}
The first inequality is satisfied for some $g_0\in\G$ because the action is continuous and $\G$ is compact.
\end{proof}

Now we are ready to prove our Theorem \ref{tedr}.

\begin{proof}[Proof of the Equivariant De Rham's Theorem]
	Let $\al$ be the action of $\G$ on $\M$. The first thing we have to do is to build a special cover since we are going to use the operator $\Zo$ defined in Theorem \ref{teoderham} which actually depends on the cover. By the Tubular Neighborhood Theorem (see \cite{alexandrino_bettiol2015}), given a point $p\in\M$ and its orbit $\G(p)$, the tubular neighborhood of this orbit can be seen as 
	\begin{equation*}
		\text{Tub}(\G(p))=\bigcup_{g\in\G}\al(g,V)=\bigcup_{g\in\G}g V
	\end{equation*}
	where $V$ is a neighborhood of $p$ such that $V$ is a smooth coordinate domain whose image under a smooth coordinate map $\vf$ is $\B^{n}\subset\R^{n}$ for some $r<1$ and $\bar{U}\subset V$:
	\begin{equation*}
		\vf(U)=B_{r}(0),\quad\vf(\bar{U})=\bar{B}_{r}(0)\quad\text{ and }\quad\vf(V)=\B^{n}.
	\end{equation*}
We say that such tubular neighborhood 
\begin{equation*}
	\bigcup_{g\in\G}g V
\end{equation*}
is generated by $V$. Finally, we get our special tubular cover built by the previous process:
\begin{equation*}
	\left\lbrace \bigcup_{g\in\G}g V_i\right\rbrace_ {i\in\N}.
\end{equation*}

Now we want to apply the smoothing operator on one of these sets $g V_{i}$. We use the corresponding coordinate chart $\vf_{i}:V_{i}\to \R^{n}$ and a nonnegative function $h_{i}$ of class $C^{\infty}$ with support in $V_{i}$ that is equal to 1 in a neighborhood of $\bar{U}_{i}$ contained in $V_{i}$. If $T\in\Ddd_{m}(\M)$, then we set $T'=h_{i}T$ and $T''=T-T'$. We note that $T'$ is an $m$-current that is equal to $T$ in $\bar{U}_{i}$ and its support is contained in $V_{i}$. Then we define 
\begin{eqnarray}
\left. \Zo\right|_{gV_{i}}T(\om)&:=&\left( \al^{g}\circ\vf_{i}^{-1} \right)\sfo \ZZo\left( \vf_{i}\circ\left(\al^{g} \right)^{-1}  \right) \sfo\left( \al^{g}\right)\sfo h_{i}T(\om)+ \left( \al^{g}\right)\sfo T''(\om)\nonumber\\
&=&\left( \al^{g}\circ\vf_{i}^{-1} \right)\sfo \ZZo \left( \vf_{i}\circ\left(\al^{g} \right)^{-1}  \right) \sfo\left( h_{i}\circ\left( \al^{g}\right)^{-1}  \right) \left( \al^{g}\right)\sfo T(\om)\nonumber\\
& &+ \left( \al^{g}\right)\sfo T''(\om)\nonumber\\
&=&\left( \al^{g}\circ\vf_{i}^{-1} \right)\sfo \ZZo\left( h_{i}\circ\left( \al^{g}\right)^{-1}  \right)\left( \vf_{i}\circ\left( \al^{g}\right)^{-1} \right)^{-1}  \left( \vf_{i}\circ\left(\al^{g} \right)^{-1}  \right) \sfo\left( \al^{g}\right)\sfo T(\om) \nonumber\\
& &+ \left( \al^{g}\right)\sfo T''(\om)\nonumber\\
&=&\left( \al^{g} \right)\sfo \left( \vf^{-1}_{i} \right)\sfo\ZZo\left( h_{i}\circ\vf^{-1}_{i} \right)\left( \vf_{i}\right)\sfo T  (\om)     + \left( \al^{g}\right)\sfo T''(\om)\nonumber\\
&=&\left( \al^{g} \right)\sfo \left( \vf^{-1}_{i} \right)\sfo\ZZo \left( \vf_{i}\right)\sfo h_{i}T(\om)       + \left( \al^{g}\right)\sfo T''(\om)\nonumber\\
&=:& \left( \al^{g} \right)\sfo\left.\Zo\right|_{V_{i}} T(\om),\label{opinvra}
\end{eqnarray}
for every $\om\in\Ddd^{m}(\M)$, using the following equation
\begin{equation*}
	T=(\al^{g})\sfo T=(\al^{g})\sfo (T'+T'')=(\al^{g})\sfo T' +(\al^{g})\sfo T''
\end{equation*}
and noticing that $(\al^{g})\sfo T'$ is an $m$-current that is equal to $T$ in $g\bar{U}_{i}$ and its support is contained in $g V_{i}$ by construction.

After all the calculations in (\ref{opinvra}) we can define an operator on the tube $\bigcup_{g\in\G}g V_i$ using Haar's Theorem (see Theorem 3.1 of \cite{bredon1972}) as follows:
\begin{equation*}
	\Zo^{\G}_{i}T(\om)=\frac{1}{\left| \G\right|}\int_{\G} \left( \al^{g} \right)\sfo\left.\Zo\right|_{V_{i}} T(\om)\,\ud g,
\end{equation*}
where $\left| \G\right|$ denotes the volume of $\G$. Finally, in the same way as in the proof of de Rham's approximation theorem we can define the operator $\Zo^{\G}:\Ddd_{m}(\M)\to\Ddd_{m}(\M)$:
\begin{equation*}
\Zo^{\G}T(\om)=\lim_{l\to\infty}\Zo^{\G}_{l}\circ\cdots\Zo^{\G}_{2}\circ\Zo^{\G}_{1}T(\om).
\end{equation*}

Thanks to the properties of the Haar integral the limit exists and it is invariant by construction.

By using standard integration theory techniques, the compactness of $\G$ and Proposition 3.2 of \cite{bredon1972}, we obtain that $\Zo ^{\G}T$ is a current of class $C^{\infty}$.

Finally, $\Zo^{\G}$ weakly converges because it converges at every step of the construction and we use the properties of the integral and Proposition \ref{proinv}:
\begin{eqnarray*}
\left|\Zo^{\G}_{i}T(\om_k)\right|&=&\left|\frac{1}{\left|\G\right|}\int_{\G} \left( \al^{g} \right)\sfo\left.\Zo\right|_{V_{i}} T(\om_k)\,\ud g\right|\leq\frac{1}{\left|\G\right|} \int_{\G}\left| \left( \al^{g} \right)\sfo\left.\Zo\right|_{V_{i}} T(\om_k)\right|\,\ud g\\
&\leq& C\left| \left.\Zo\right|_{V_{i}} T(\om_k) \right|\to C \left| T(\om)\right|  \text{ as }k\to\infty,
\end{eqnarray*} 
from which we deduce the weakly convergence of $\Zo^\G$.
\end{proof}

\section{Spaces with bounded curvature}
We introduce some basic notation and definitions in preparation to prove the equivariant version of Nikolaev's Theorem. Our proof is based on Nikolaev's proof in \cite{nikolaev1991}. We can see the details about Lipschitz convergence in \cite{burago_burago_ivanov2001}.

If $U\subset \R^{n}$, we denote by $\Mr^{2,p}(U)$ the space of Riemannian metrics
\begin{equation*}
\gr(x)=\left(\gr_{ij}(x) \right)_{i,j=1,\ldots,n}
\end{equation*}
 which are twice continuously differentiable almost everywhere on the domain $U$ with $1\leq p\leq\infty$ for which the following norm is finite:
 \begin{equation*}
 	\left|\gr \right|_{\Mr^{2,p}(U)}=\max_{i,j}\left\lbrace  \left| g_{ij}\right|_{W^{2,p}(U)} \right\rbrace. 
 \end{equation*}

Let $\M$ be a differentiable manifold with fixed  $C^{\infty}$-atlas $\hr=\left\lbrace (U_{i},\vf_{i}) \right\rbrace_{i\in\N}$. By $\Mr^{2,p}_{\hr}(\M)$ we denote the set of continuous Riemannian metrics $\gr$ on $\M$ for which the following seminorms are finite:
\begin{equation*}
	\left| \gr \right|_{\Mr^{2,p}_{\hr}(\M),i}=\left| \left(\vf^{-1}_{i} \right)\s\gr \right|_{\Mr^{2,p}(V_{i})},\quad V_{i}=\vf_{i}(U_{i}),\quad i\in\N.
\end{equation*}

One defines spaces $\Mr^{r,\al}_{\hr}(\M)$, $r\in\N$, $0<\al<1$, with seminorms $\left| \gr \right|_{\Mr^{r,\al}_{\hr}(\M),i}$ analogously using $C^{r,\al}(U)$ instead of $W^{2,p}(U)$. 

We denote by $\Mr^{\infty}_{\hr}(\M)$ the set of smooth Riemannian metrics on $\M$ which are infinitely differentiable with respect  to the atlas $\hr$.

Let $\gr\in\Mr^{2,p}_{\hr}(\M)$. Then at almost each point $p\in\M$ and section $\si\subset T_{p}\M$ we can formally calculate the sectional curvature $K_{\si}(p)$ with respect to the Riemannian metric $\gr$, by using the Christoffel symbols and the fact that our metric has second derivatives almost everywhere. By $\underaccent{\bar}{K}_{\gr,\mu}$, $\bar{K}_{\gr,\mu}$, we denote the essential infimum and essential supremum, respectively, of $K_{\si}(p)$ for every point $p\in\M$ and every section $\si\subset T_{p}\M$.

We are ready to see the equivariant approximation process of spaces with bounded curvature, Theorem \ref{ENATeo}.

\begin{proof}[Proof of the Equivariant Nikolaev's Approximation Theorem]
	We will use the same cover constructed in the proof of Theorem \ref{tedr}: 
	\begin{equation*}
		\left\lbrace \bigcup_{g\in\G}g V_i\right\rbrace_ {i\in\N}.
	\end{equation*}
As before, we want to define the regularization operator of metrics $\Hr$ in $g V_{i}$. To do this, we use the corresponding coordinate chart $\vf_{i}:V_{i}\to \R^{n}$ and a nonnegative function $h_{i}$ of class $C^{\infty}$ with support in $V_{i}$ that is equal to 1 on a neighborhood of $\bar{U}_{i}$ contained in $V_{i}$. Then we decompose $\gr_{0}$ into two smooth metric tensors $\gr_{0}=\gr'_{i}+\gr''_{i}$, where $\gr'_{i}=h_{i}\,\gr_{0}$ and $\gr''_{i}=\gr_{0}-\gr'_{i}$. We let
\begin{equation*}
	\Hr_{e,i}(\gr_{0})=\left( \al^{e}\circ\vf_{i}^{-1} \right)\s \Hrt_\e \left( \vf_{i}\circ\left(\al^{e} \right)^{-1}  \right) \s\left( \al^{e}\right)\s h_{i}\,\gr_{0}+ \left( \al^{e}\right)\s \gr''_{i},
\end{equation*}
where 
\begin{equation*}
	\Hrt_\e(\gr)=\int_{\R^{n}}f_{\e}(y)\cdot(s_{y})\s(\gr)\ud y
\end{equation*}
for a Riemannian metric $\gr$ on $\R^{n}$ and $e\in\G$ is the identity element. We note that 
\begin{equation*}
	\gr_{e,i}(x):=\left.\Hr_{e,i}(\gr_{0})\right|_{x}=\gr_{0}(x)
\end{equation*}
for $x\in\M\smallsetminus  V_{i}$ and it follows from the last two equations that $\gr_{g,i}$ is a Riemannian metric on $\M$. This is the regularization process of the metric on $V_{i}$. In order to make this process in $g V_{i}$, we take the decomposition $\gr_{0}=\left( \al^{g}\right)\s \gr_{0}=\left( \al^{g}\right)\s\gr_i'+\left( \al^{g}\right)\s\gr_i''$ and then
\begin{equation*}
	\Hr_{g,i}(\gr_{0}):=\left( \al^{g}\circ\vf_{i}^{-1} \right)\s \Hrt_{\e} \left( \vf_{i}\circ\left(\al^{g} \right)^{-1}  \right) \s\left( \al^{g}\right)\s h_{i}\,\gr_{0}+ \left( \al^{g}\right)\s \gr''_{i}=\left( \al^{g}\right)\s\Hr_{e,i}(\gr_{0}).
\end{equation*}
As in the proof of Theorem \ref{tedr}, we define a metric on the tube 
\begin{equation*}
	\bigcup_{g\in\G}g V_i
\end{equation*}
using Haar's Theorem (see Theorem 3.1 of \cite{bredon1972}) by letting
\begin{equation*}
	\Hr^{\G}_{i}(\gr_{0})=\frac{1}{\left|\G\right|}\int_{\G} \left( \al^{g} \right)\s \Hr_{e,i}(\gr_{0}) \ud g.
\end{equation*}
Finally, also as in the proof of Theorem \ref{tedr}, we define the operator 
\begin{equation*}
\Hr^{\G}:\Mr^{2,p}_{\hr}(\M)\to\Mr^{\infty}_{\hr}(\M)
\end{equation*}
 by setting:
\begin{equation*}
	\Hr^{\G}(\gr_0)=\lim_{s\to\infty}\Hr^{\G}_{s}\circ\cdots\circ\Hr^{\G}_{2}\circ\Hr^{\G}_{1}(\gr_{0}),
\end{equation*}
where $\hr$ is the atlas defined above. Since $\M$ is compact, the limit exists, i.e. we finish the regularization process and obtain a $\G$-invariant Riemannian metric on $\M$. To see this fact we take coordinates $\vf_i$ in $g V_i$ for some $i\in\N$ and $g\in\G$, and coordinate vector fields $\{\partial_j\}_{j=1}^{n}$, then

\begin{equation*}
	\Hr^{\G}(\gr_{0})(\partial_l,\partial_m)=\Hr^{\G}(\gr_{0})_{l\,m}=\Hr^{\G}_{s'_i}\circ\cdots\circ\Hr^{\G}_{s_i}(\gr_{0})_{l\,m}=\Hr_{e,i}(\gr_{0})_{l\,m}=(\gr_{0})_{l\,m}
\end{equation*}

Also we notice that 
\begin{equation*}
	\left| \Hr^{\G}_{i}(\gr_{0})_{l\,m} \right| = \left|\int_{\G} \left( \al^{g} \right)\s \Hr_{e,i}(\gr_{0})_{l\,m} \ud g \right|\leq \int_{\G} \left|  \left( \al^{g} \right)\s \Hr_{e,i}(\gr_{0})_{l\,m} \right|  \ud g
\end{equation*}
and since $\G$ is compact there exists $\tilde{g}\in\G$ such that 
\begin{equation*}
	\int_{\G} \left|  \left( \al^{g} \right)\s \Hr_{e,i}(\gr_{0})_{l\,m} \right|  \ud g	\leq \int_{\G} \left|  \left( \al^{\tilde{g}} \right)\s \Hr_{e,i}(\gr_{0})_{l\,m} \right|  \ud g.
\end{equation*}
Now, since $\M$ is compact, we get a constant $C$ such that
\begin{equation*}
	\int_{\G} \left|  \left( \al^{\tilde{g}} \right)\s \Hr_{e,i}(\gr_{0})_{l\,m} \right|  \ud g\leq C\int_{\G} \left|   \Hr_{e,i}(\gr_{0})_{l\,m} \right|  \ud g.
\end{equation*}
Therefore,
\begin{equation}\label{gequ}
	\left| \Hr^{\G}_{i}(\gr_{0})_{l\,m} \right|\leq C \left| \G\right| \,\left|   \Hr_{e,i}(\gr_{0})_{l\,m} \right|.
\end{equation}

To continue with the proof we need the following couple of estimates. The first one is exactly the same as Lemma 3.1 of \cite{nikolaev1991}. The second one is similar to Lemma 3.2 also from \cite{nikolaev1991}.  

\begin{lem}[Lemma 3.1 of \cite{nikolaev1991}]\label{lemma3.1}
	Let $U\subset \R^{n}$ be a domain. The operator $\Hrt_{\e}$ maps ${\Mr^{2,p}(U)}$ into $\Mr^{2,p}(U)$ for any $1\leq p \leq \infty$ while for each positive number $\de>0$ we can find a $\nu_{\de,p}>0$ such that for all $0<\e<\nu_{\de,p}$ we have:
	\begin{enumerate}
		\item For $1\leq p\leq\infty$ we have the estimate
		\begin{equation}\label{eq1lemm3.1}
			\left| \Hrt_\e(\textbf{\gr})-\textbf{\gr} \right|_{\Mr^{2,p}(U)}\leq\de.
		\end{equation}
		\item Let $K_{\si}(x)$, $K^{\e}_{\si}(x)$ be the sectional curvatures calculated from the Riemannian metrics $\textbf{\gr}\in\Mr^{2,p}(U)$ and $\textbf{\gr}_{\e}=\Hrt_{\e}(\textbf{\gr})$, respectively, where $x\in U$ and $\si\subset T_{x}U$ is a section (i.e. a $2$-dimensional subspace). Then if 
		\begin{eqnarray*}
			-\infty<&\bar{K}_{\si}(U)=\esssup_{x\in U}\{K_{\si}(x)\} &<\infty,\\
			-\infty<&\underaccent{\bar}{K}_{\si}(U)=\essinf_{x\in U}\{K_{\si}(x)\} &<\infty,
		\end{eqnarray*}
		then under the condition that $p>n$ the same bounds are true for the corresponding quantities $\bar{K}^{\e}_{\si}(U)$, $\underaccent{\bar}{K}^{\e}_{\si}$ calculated from the metric $\textbf{\gr}_{\e}$ and for each case the following holds:
		\begin{equation}\label{eq2lemm3.1}
			\left|\bar{K}_{\si}(U)-\bar{K}^{\e}_{\si}(U)\right|,\left|\underaccent{\bar}{K}_{\si}(U)-\underaccent{\bar}{K}^{\e}_{\si}\right|<\delta.
		\end{equation}
	\end{enumerate}
\end{lem}

The proof of the previous lemma is based on an estimate of the application $s_y$ defined in (\ref{aplis}):
\begin{equation*}
	\left| s_{y\e}- \id_{U} \right|_{C^3(U)} \leq c_{U}(\e),
\end{equation*}
where $\left| y \right|\leq 1 $ and $c_{U}(\e)\to0$ as $\e\to 0$, which follows from the fact that $s_y$ is a diffeomorphism equal to 1 outside of the unit ball and the usual way of computing the sectional curvature.

\begin{lem}[Lemma 3.2 of \cite{nikolaev1991}]\label{lemma3.2}
	Let $\M$ be a differentiable manifold, $\hr$ be a $C^{\infty}$-differentiable locally finite countable atlas on $\M$ with the help of the operator $\Hr^{\G}$ constructed before. The following assertions hold:
	\begin{enumerate}
		\item For an arbitrary bounded sequence of numbers $\e_{i}$, $i=1,2,\ldots$, the operator $\Hr^{\G}$ maps $\Mr^{2,p}_{\hr}(\M)$ into $\Mr^{\infty}_{\hr}(\M)$.
		\item For each natural number $k$ and arbitrary sequence of positive numbers $a_\nu$, $\nu\in\N$, there exists a uniformly bounded subsequence of positive numbers $\e_{k_i}$, $i,k\in\N$, such that for $\nu\in\N$, $\textbf{\gr}\in\Mr^{2,p}_\hr(\M)$ the following estimate holds:
		\begin{equation}\label{eq1lemm3.2}
			\left| \Hr^{\G}(\textbf{\gr})-\textbf{\gr} \right|_{\Mr^{2,p}_{\hr}(\M),\nu}\leq\frac{a_\nu}{k}.
		\end{equation}
		Moreover, if $p>n$, we have the estimate for then in addition we can assert the existence of an estimate for $\al=1-n/p$
		\begin{equation}\label{eq2lemm3.2}
			\left| \Hr^{\G}(\textbf{\gr})-\textbf{\gr} \right|_{\Mr^{1,\al}_{\hr}(\M),\nu}\leq\frac{a_\nu}{k}
		\end{equation}
		and the following relations hold:
		
		\begin{eqnarray}
			\left|\bar{K}_{\textbf{\gr}_{k},\mu}-\bar{K}_{\textbf{\gr},\mu}\right|\leq\frac{1}{k}& & \text{if }\bar{K}_{\textbf{\gr},\mu}<\infty,\label{eq3lemm3.2}\\
			\bar{K}_{\textbf{\gr}_{k},\mu}=+\infty & &\text{if }\bar{K}_{\textbf{\gr},\mu}=+\infty,\label{eq4lemm3.2}
		\end{eqnarray}
		where $\textbf{\gr}_{k}=\Hr(\textbf{\gr})$ is calculated with this sequence, and
		\begin{eqnarray}
			\left|\underaccent{\bar}{K}_{\textbf{\gr}_{k},\mu}-\underaccent{\bar}{K}_{\textbf{\gr},\mu}\right|\leq\frac{1}{k}& &\text{if }\underaccent{\bar}{K}_{\textbf{\gr},\mu}>-\infty,\label{eq5lemm3.2}\\
			\underaccent{\bar}{K}_{\textbf{\gr}_{k},\mu}=-\infty & &\text{if }\underaccent{\bar}{K}_{\textbf{\gr},\mu}=-\infty.\label{eq6lemm3.2}
		\end{eqnarray}
	\end{enumerate}
\end{lem}

The proof of this lemma follows the same structure as in \cite{nikolaev1991}. Using the construction of the operator $\Hr^{\G}$, the inequality (\ref{gequ}) and (\ref{eq1lemm3.1}) of Lemma \ref{lemma3.1} we get (\ref{eq1lemm3.2}). To obtain (\ref{eq2lemm3.2}) we use Rellich-Kondrashov theorem (see \cite{nikolski1977}). Finally, the estimates about the curvature are the result of using (\ref{eq1lemm3.2}), the convergence ``almost everywhere'' of the second derivative of the pullback metric on a domain in $\R^n$ and the estimates for the curvature of Lemma \ref{lemma3.1}.

We now continue with the proof of our theorem introducing the following notation: $\bar{\gr}_{0}^{\nu}=\left.(\al^{g}\circ\vf^{-1}_{\nu})\s\gr_{0}\right|_{gV_{\nu}}$ and $\bar{\gr}_{0}^{\nu}(x)=((\bar{\gr}_{0}^{\nu})_{ij}(x))$, $i,j=1,\ldots,n$, $x\in\B^{n}$, and we define a sequence of positive numbers $\{a_\nu\}$ with the help of the following equations:
\begin{equation}\label{secu}
	a_\nu=\inf_{x\in\bar{\B}^{n}}\left\lbrace  \min_{\left|\xi\right|\neq0,\,\xi\in\R^{n}}  \left\lbrace \frac{ (\bar{\gr}_{0}^{\nu})_{ij}(x) \xi^{i}\xi^{j} }{\delta_{ij}\xi^{i}\xi^{j}} \right\rbrace  \right\rbrace
\end{equation}

By $\gr_{k}=\Hr^{\G}(\gr_{0})$ we denote the sequence of Riemannian metrics on $\M$ constructed in Lemma \ref{lemma3.2} for the sequence $\{a_{\nu}\}$ given by equation (\ref{secu}).

We also introduce the following notation: $d_{k}=d(\gr_{k})$, $d_{0}=d(\gr_{0})$ and \begin{equation*}
	\iota_{k}:(\M,d_{0})\to(\M,d_{k})
\end{equation*} 
is a map for which $\iota_{k}(p)=p$, $p\in\M$ and $k\in\N.$ By $((\bar{\gr}_{k}^{\nu})_{ij}(x))$, $i,j=1,\ldots,n$ and $x\in\B^{n}$, we denote the components of the metric tensor by
\begin{equation*}
\bar{\gr}_{k}^{\nu}(x)=\left.(\al^{g}\circ\vf^{-1}_{\nu})\s\Hr^{\G}(\gr_{0})\right|_{gV_{\nu}}(x).
\end{equation*}

Let $\ga:[0,l_{0}]\to\M$ be an arbitrary differentiable curve with respect to the atlas $\hr$ and  parametrized by arc length in $(\M,d_{0})$. Its length in $(\M,d_{k})$ will be denoted by $l_{k}$.

We divide $\ga$ into a finite number of arcs $\ga_{u}$ each of which is contained in some $g_{u}V_{\nu_{u}}$ for $u=1,\ldots,N$. We denote the length of $\ga_{u}$ in $(\M,d_{0})$ by $l_{0}^{(u)}$ and its length in $(\M,d_{k})$ by $l_{k}^{(u)}$. Then
\begin{equation}\label{eq3.12teo}
	\left| l_{k}^{(u)}-l_{0}^{(u)} \right|\leq\int_{0}^{l_{0}^{(u)}} \left| (\bar{\gr}_{k}^{\nu_u})_{mq}-(\bar{\gr}_{0}^{\nu_u})_{mq}\right|\left| \dga_u^m \dga_u^q\right|  \ud s. 
\end{equation}

From the expression for $a_{\nu}$ it follows that 
\begin{equation*}
	\max_{\substack{m,q=1,2,\ldots,n\\0\leq s\leq l_{0}^{(u)}}}\left\lbrace \left| \dga_u^m(s)\cdot \dga_u^q(s)\right| \right\rbrace \leq \max_{\substack{m,q=1,2,\ldots,n\\0\leq s\leq l_{0}^{(u)}}}\left\lbrace \left| \de_{mq}\dga_u^m(s)\cdot \dga_u^q(s)\right| \right\rbrace \leq a^{-1}_{\nu_{u}}
\end{equation*}

If follows from the last inequality, inequality (\ref{eq3.12teo}) and inequality (\ref{eq2lemm3.2}) of Lemma \ref{lemma3.2} that 
\begin{equation*}
	\left| l_{k}^{(u)}-l_{0}^{(u)} \right|\leq\frac{l_{0}^{(u)}}{k}.
\end{equation*}

Adding the last inequalities for $u=1,\ldots N$, we get
\begin{equation*}
	\left| l_{k}-l_0\right|\leq \frac{l_0}{k}. 
\end{equation*}

From this inequality, we obtain
\begin{equation}\label{eq3.13teo}
	\left| \frac{d_{k}(p,p')}{d_{0}(p,p')}-1\right| \leq\frac{1}{k}
\end{equation}

It follows from inequality (\ref{eq3.13teo}) that 
\begin{equation*}
	\lim_{k\to\infty}\text{dil}\,\iota_{k}=\lim_{k\to\infty}\text{dil}\,\iota^{-1}_{k}=1.
\end{equation*}
Therefore, we have assertion (2).

To prove assertion (3) we remember that for $\gr\in\Mr^{2,p}_{\hr}$ we have introduced $\underaccent{\bar}{K}_{\gr,\mu}(\M)$ and $\bar{K}_{\gr,\mu}(\M)$. We denote the corresponding quantities for the Riemannian metrics $\gr_{k}$, $k\in\N,$ by $\underaccent{\bar}{K}_{k,\mu}(\M)$ and $\bar{K}_{k,\mu}(\M)$, respectively. We note that since $\gr_{k}$ is an infinitely differentiable metric on $\M$ we have that $\underaccent{\bar}{K}_{k,\mu}(\M)=\underaccent{\bar}{k}_{k}(\M)$ and $\bar{K}_{k,\mu}(\M)=\bar{k}_{k}(\M)$. Considering inequalities (\ref{eq3lemm3.2})-(\ref{eq6lemm3.2}) of Lemma \ref{lemma3.2}, to prove assertion (3) it remains to note that we have:
\begin{equation*}
	\bar{K}_{0,\mu}(\M) \leq\bar{k}_{0}(\M) \text{ and }\underaccent{\bar}{K}_{0,\mu}(\M)\geq \underaccent{\bar}{k}_{0}(\M)
\end{equation*}
which follow directly from Theorem 2.1 of \cite{nikolaev1991}.

\end{proof}

\section{Equivariant Sphere Theorem and other applications}
Here we establish an algorithm using Theorem \ref{ENATeo} that extends to compact spaces with bounded curvature endowed with a compact Lie group action any result proved for compact Riemannian manifolds endowed with an action of the same group.
As a first application of this algorithm, we prove Theorem \ref{teoesf}. 

\begin{proof}[Proof of the Equivariant Sphere Theorem]
	Theorem \ref{ENATeo} guarantees that for each $\e>0$ we can find a Riemannian manifold $(\M,\gr_{\e})$ of class $C^{\infty}$ that is $d_{L}$-close  to the original metric space and whose sectional curvatures for all point $p\in\M$ and section $\pi\subset T_p\M$ satisfies
	\begin{equation*}
		\underaccent{\bar}{K}(\M)-\e\leq K_{\pi}(p)\leq 1+\e.
	\end{equation*}
Also we have that $\G$ acts by isometries on $(\M,\gr_{\e})$.

We choose $\e$ in such a way that 
\begin{equation*}
	\frac{\underaccent{\bar}{K}(\M)-\e}{1+\e}>\frac{1}{4}.
\end{equation*}
Multiplying $\gr_{\e}$ by a constant we can assume that 
\begin{equation*}
	\frac{1}{4}<c\leq K_{\pi}(p)\leq 1,
\end{equation*}
where
\begin{equation*}
	c=\frac{\underaccent{\bar}{K}(\M)-\e}{1+\e}.
\end{equation*}
Therefore, by virtue of the smooth Equivariant Sphere Theorem (see Theorem 2 of \cite{brendle_schoen2009}), we obtain the result.
\end{proof}
\begin{obss}
    If we change (\ref{eqesf}) to be 
    \begin{equation}\label{eqesf2}
        \frac{1}{4}\leq\underaccent{\bar}{K}(\M)\leq \bar{K}(\M)\leq 1,
    \end{equation}
    we can improve Nikolaev's Sphere Theorem (Theorem 3.2 of \cite{nikolaev1991}) following the same structure as before but using Theorem 1.1 of \cite{petersen_tao2009}, a classification of almost quarter-pinched manifolds. In this case, we obtain that $\M$ is diffeomorphic to a sphere or compact rank one symmetric space.
\end{obss}
 A different way to obtain a compact rank one symmetry space is with our second application, Theorem \ref{symspa}. 
\begin{proof}[Proof of Theorem \ref{symspa}]
Let $\M$ be a even dimensional compact, simply connected cohomogeneity one $\G$-space with bounded curvature such that the lower bound is greater than zero. Also assume that the $\G$ is compact.  By Theorem \ref{ENATeo} as before, we get a sequence of even dimensional compact, simply connected cohomogeneity one $\G$-manifolds with positive sectional curvature.

Using Theorem 1.1 of \cite{verdiani2004}, we obtain that $\M$ is equivariently diffeomorphic to a compact rank one symmetry space.
\end{proof}
Finally we prove the last application.
\begin{proof}[Proof of Theorem \ref{cohomogteo}]
    Any compact cohomogeneity one space with bounded curvature can be approximated by a sequence of compact cohomogeneity on Riemannian manifolds using Theorem \ref{ENATeo} with the curvature bounds are preserved. Furthermore, all the elements of the sequence have the same topology and they are compact cohomogeneity one manifolds with positive sectional curvature. Since a lower sectional curvature bound implies a lower Ricci curvature bound, we can apply Theorem A of \cite{grove_ziller2002} and obtaining the result.
\end{proof}

\begin{obss}[Homogeneous spaces]
The equivariant smoothing procedure of Theorem \ref{ENATeo} also applies in the homogeneous setting. If $(\M,d)$ is a compact, simply connected space with bounded curvature admitting a transitive isometric action of a compact Lie group $\G$, so that $\M \cong \G/\Hr$, then Theorem \ref{ENATeo} produces smooth $\G$--invariant Riemannian metrics $\textbf{\gr}_k \in \Mr^{\infty}_{\hr}(\M)$ whose induced distances converge to $d$ in the Lipschitz sense while preserving the two--sided curvature bounds. In particular, if $(\M,d)$ has a positive lower curvature bound, then for $k$ large the manifolds $(\M,\textbf{\gr}_k)$ are simply connected compact homogeneous spaces with positive sectional curvature. Hence the classification results of Wallach in \cite{wallach1972} and B\'erard--Bergery in \cite{berard-bergery1976} apply and restrict the possible homogeneous spaces $\G/\Hr$ arising in the bounded--curvature context.
\end{obss}

\bibliographystyle{plain} 
\bibliography{biblio}

@Book{burago_burago_ivanov2001,
    Author = {D. {Burago} and Yu. {Burago} and S. {Ivanov}},
    Title = {{A course in metric geometry.}},
    FJournal = {{Graduate Studies in Mathematics}},
    Journal = {{Grad. Stud. Math.}},
    ISSN = {1065-7338},
    Volume = {33},
    ISBN = {0-8218-2129-6/hbk},
    Pages = {xiv + 415},
    Year = {2001},
    Publisher = {Providence, RI: American Mathematical Society (AMS)},
    Language = {English},
    MSC2010 = {51K05 51-01},
    Zbl = {0981.51016}
}

@Article{berestovski1976,
	Author = {V. N. {Berestovskij}},
	Title = {{Introduction of a Riemann structure into certain metric spaces.}},
	FJournal = {{Siberian Mathematical Journal}},
	Journal = {{Sib. Math. J.}},
	ISSN = {0037-4466; 1573-9260/e},
	Volume = {16},
	Pages = {499--507},
	Year = {1976},
	Publisher = {Springer US, New York, NY; Pleiades Publishing, New York, NY; MAIK ``Nauka/Interperiodica'', Moscow},
	Language = {English},
	MSC2010 = {53C70},
	Zbl = {0325.53059}
}

@Article{nikolaev1983,
	Author = {I. G. {Nikolaev}},
	Title = {{On the parallel displacement of vectors in spaces with two-sided bounded curvature in the sense of A. D. Aleksandrov.}},
	FJournal = {{Siberian Mathematical Journal}},
	Journal = {{Sib. Math. J.}},
	ISSN = {0037-4466; 1573-9260/e},
	Volume = {24},
	Pages = {106--119},
	Year = {1983},
	Publisher = {Springer US, New York, NY; Pleiades Publishing, New York, NY; MAIK ``Nauka/Interperiodica'', Moscow},
	Language = {English},
	MSC2010 = {53C20},
	Zbl = {0539.53030}
}

@Misc{nikolski1977,
	Author = {S. M. {Nikol'skij}},
	Title = {{Approximation of functions of several variables and imbedding theorems. (Priblizhenie funktsij mnogikh peremennykh i teoremy vlozheniya). 2nd ed., rev. and suppl.}},
	Year = {1977},
	Language = {Russian},
	HowPublished = {{Moskva: ''Nauka''. 455 p.}},
	MSC2010 = {46E35 46-02 41A65},
	Zbl = {0496.46020}
}

@Misc{aleksandrov1957,
	Author = {A. D. {Alexandrov}},
	Title = {{\"Uber eine Verallgemeinerung der Riemannschen Geometrie.}},
	Year = {1957},
	Language = {German},
	HowPublished = {{Begriff des Raumes in der Geometrie. Ber. Riemann-Tagung Forsch.-Inst. Math. 33-84 (1957).}},
	Zbl = {0077.35702}
}

@Article{nikolaev19832,
	Author = {I. G. {Nikolaev}},
	Title = {{Smoothness of the metric of spaces with two-sided bounded Aleksandrov curvature.}},
	FJournal = {{Siberian Mathematical Journal}},
	Journal = {{Sib. Math. J.}},
	ISSN = {0037-4466; 1573-9260/e},
	Volume = {24},
	Pages = {247--263},
	Year = {1983},
	Publisher = {Springer US, New York, NY; Pleiades Publishing, New York, NY; MAIK ``Nauka/Interperiodica'', Moscow},
	Language = {English},
	MSC2010 = {53B20},
	Zbl = {0547.53011}
}

@Book{morgan16,
	Author = {Frank {Morgan}},
	Title = {{Geometric measure theory. A beginner's guide. Illustrated by James F. Bredt. 5th edition}},
	Edition = {5th edition},
	ISBN = {978-0-12-804489-6},
	Pages = {viii + 263},
	Year = {2016},
	Publisher = {Amsterdam: Elsevier/Academic Press},
	Language = {English},
	MSC2010 = {49Q15 49-01 49Q05 49Q20 53-01 53C42},
	Zbl = {1338.49089}
}

@Book{derham1984,
	Author = {Georges {de Rham}},
	Title = {{Differentiable manifolds. Forms, currents, harmonic forms. Transl. from the French by F. R. Smith. Introduction to the English ed. by S. S. Chern}},
	FJournal = {{Grundlehren der Mathematischen Wissenschaften}},
	Journal = {{Grundlehren Math. Wiss.}},
	ISSN = {0072-7830; 2196-9701/e},
	Volume = {266},
	Year = {1984},
	Publisher = {Springer, Berlin},
	Language = {English},
	MSC2010 = {58Axx 58-02 57-02 55Nxx 58A05 58A10 58A12 58A14 58A25},
	Zbl = {0534.58003}
}

@Book{federer96,
	Author = {Herbert {Federer}},
	Title = {{Geometric measure theory. Repr. of the 1969 ed}},
	FJournal = {{Classics in Mathematics}},
	Journal = {{Class. Math.}},
	ISSN = {1431-0821},
	Edition = {Repr. of the 1969 ed.},
	ISBN = {3-540-60656-4},
	Pages = {xvi + 680},
	Year = {1996},
	Publisher = {Berlin: Springer-Verlag},
	Language = {English},
	MSC2010 = {49-01 49Q15 28A75 58A15 58A25 49Q20 58C35},
	Zbl = {0874.49001}
}

@article{nikolaev1991,
	Title = {{Closure of the set of classical Riemannian spaces.}},
	Author = {I. G. {Nikolaev}},
	Year = {1991},
	doi = "10.1007/BF01095905",
	Language = {English},
	Volume = {55},
	Pages = {2100--2115},
	Journal = {{Journal of Mathematical Sciences}},
	ISSN = {1072-3374},
	Publisher = {Springer},
	number = "6",
}

@Book{alexandrino_bettiol2015,
	Author = {Alexandrino, Marcos M. and Bettiol, Renato G.},
	Title = {Lie groups and geometric aspects of isometric actions},
	ISBN = {978-3-319-16612-4; 978-3-319-16613-1},
	Year = {2015},
	Publisher = {Cham: Springer},
	Language = {English},
	DOI = {10.1007/978-3-319-16613-1},
	zbMATH = {6423105},
	Zbl = {1322.22001}
}

@Misc{bredon1972,
	Author = {Bredon, Glen E.},
	Title = {Introduction to compact transformation groups},
	Year = {1972},
	Language = {English},
	HowPublished = {Pure and {Applied} {Mathematics}, 46. {New} {York}-{London}: {Academic} {Press}. {XIII},459 p.},
	zbMATH = {3390006},
	Zbl = {0246.57017}
}

@Article{brendle_schoen2009,
	Author = {Brendle, Simon and Schoen, Richard},
	Title = {Manifolds with {{\(1/4\)}}-pinched curvature are space forms},
	FJournal = {Journal of the American Mathematical Society},
	Journal = {J. Am. Math. Soc.},
	ISSN = {0894-0347},
	Volume = {22},
	Number = {1},
	Pages = {287--307},
	Year = {2009},
	Language = {English},
	DOI = {10.1090/S0894-0347-08-00613-9},
	zbMATH = {5859406},
	Zbl = {1251.53021}
}

@article {myers_steenrod1939,
	AUTHOR = {Myers, S. B. and Steenrod, N. E.},
	TITLE = {The group of isometries of a {R}iemannian manifold},
	JOURNAL = {Ann. of Math. (2)},
	FJOURNAL = {Annals of Mathematics. Second Series},
	VOLUME = {40},
	YEAR = {1939},
	NUMBER = {2},
	PAGES = {400--416},
	ISSN = {0003-486X},
	MRCLASS = {DML},
	MRNUMBER = {1503467},
	DOI = {10.2307/1968928},
	URL = {https://doi.org/10.2307/1968928},
}

@article {fukaya_yamaguchi1994,
	AUTHOR = {Fukaya, Kenji and Yamaguchi, Takao},
	TITLE = {Isometry groups of singular spaces},
	JOURNAL = {Math. Z.},
	FJOURNAL = {Mathematische Zeitschrift},
	VOLUME = {216},
	YEAR = {1994},
	NUMBER = {1},
	PAGES = {31--44},
	ISSN = {0025-5874},
	MRCLASS = {53C23 (53C20)},
	MRNUMBER = {1273464},
	MRREVIEWER = {Zhongmin Shen},
	DOI = {10.1007/BF02572307},
	URL = {https://doi.org/10.1007/BF02572307},
}

@Book{petersen2016,
	Author = {Petersen, Peter},
	Title = {Riemannian geometry},
	Edition = {3rd edition},
	FSeries = {Graduate Texts in Mathematics},
	Series = {Grad. Texts Math.},
	ISSN = {0072-5285},
	Volume = {171},
	ISBN = {978-3-319-26652-7; 978-3-319-26654-1},
	Year = {2016},
	Publisher = {Cham: Springer},
	Language = {English},
	DOI = {10.1007/978-3-319-26654-1},
	zbMATH = {6520113},
	Zbl = {1417.53001}
}

@Article{petersen_tao2009,
 Author = {Petersen, Peter and Tao, Terence},
 Title = {Classification of almost quarter-pinched manifolds},
 FJournal = {Proceedings of the American Mathematical Society},
 Journal = {Proc. Am. Math. Soc.},
 ISSN = {0002-9939},
 Volume = {137},
 Number = {7},
 Pages = {2437--2440},
 Year = {2009},
 Language = {English},
 DOI = {10.1090/S0002-9939-09-09802-5},
 zbMATH = {5567253},
 Zbl = {1168.53020}
}

@article{grove_ziller2002,
 author = {Grove, Karsten and Ziller, Wolfgang},
 title = {Cohomogeneity one manifolds with positive {Ricci} curvature},
 fjournal = {Inventiones Mathematicae},
 journal = {Invent. Math.},
 issn = {0020-9910},
 volume = {149},
 number = {3},
 pages = {619--646},
 year = {2002},
 language = {English},
 doi = {10.1007/s002220200225},
 zbMATH = {1965434},
 Zbl = {1038.53034}
}

@article{verdiani2004,
 author = {Verdiani, Luigi},
 title = {Cohomogeneity one manifolds of even dimension with strictly positive sectional curvature},
 fjournal = {Journal of Differential Geometry},
 journal = {J. Differ. Geom.},
 issn = {0022-040X},
 volume = {68},
 number = {1},
 pages = {31--72},
 year = {2004},
 language = {English},
 doi = {10.4310/jdg/1102536709},
 zbMATH = {5033779},
 Zbl = {1100.53033}
}

@article{wallach1972,
 author = {Wallach, Nolan R.},
 title = {Compact homogeneous {Riemannian} manifolds with strictly positive curvature},
 fjournal = {Annals of Mathematics. Second Series},
 journal = {Ann. Math. (2)},
 issn = {0003-486X},
 volume = {96},
 pages = {277--295},
 year = {1972},
 language = {English},
 doi = {10.2307/1970789},
 zbMATH = {3411820},
 Zbl = {0261.53033}
}

@article{berard-bergery1976,
 author = {B{\'e}rard Bergery, Lionel},
 title = {Les vari{\'e}t{\'e}s riemanniennes homogenes simplement connexes de dimension impaire {\`a} courbure strictement positive},
 fjournal = {Journal de Math{\'e}matiques Pures et Appliqu{\'e}es. Neuvi{\`e}me S{\'e}rie},
 journal = {J. Math. Pures Appl. (9)},
 issn = {0021-7824},
 volume = {55},
 pages = {47--68},
 year = {1976},
 language = {French},
 zbMATH = {3454505},
 Zbl = {0289.53037}
}

\end{document}